\documentclass[12pt,reqno]{amsart}

\usepackage{color}
\usepackage{amsmath, amsthm, amssymb}
\usepackage{amsfonts}
\usepackage[ansinew]{inputenc}
\usepackage[dvips]{epsfig}
\usepackage{graphicx}
\usepackage[english]{babel}
\usepackage{hyperref}
\theoremstyle{plain}
\newtheorem{thm}{Theorem}[section]

\newtheorem{lem}[thm]{Lemma}

\theoremstyle{definition}
\newtheorem{defi}[thm]{Definition}

\theoremstyle{remark}
\newtheorem{rem}[thm]{Remark}

\numberwithin{equation}{section}

\newcommand{\average}{{\mathchoice {\kern1ex\vcenter{\hrule height.4pt
width 6pt depth0pt} \kern-9.7pt} {\kern1ex\vcenter{\hrule
height.4pt width 4.3pt depth0pt} \kern-7pt} {} {} }}

\def\R{\mathbb{R}}

\begin{document}

\title[Free boundary regularity in the $p$-Laplacian obstacle problem]{On the regularity of the free boundary\\ in the $p$-Laplacian obstacle problem}

\author{Alessio Figalli}
\address{ETH Z\"urich, Department of Mathematics, Raemistrasse 101, 8092 Z\"urich, Switzerland}
\email{alessio.figalli@math.ethz.ch}

\author{Brian Krummel}
\address{University of California, Berkeley, Department of Mathematics, 970 Evans Hall, Berkeley, CA 94720, USA}
\email{bkrummel@math.berkeley.edu}
%\address{The University of Texas at Austin, Department of Mathematics, 2515 Speedway, Austin, TX 78751, USA}
%\email{bkrummel@math.utexas.edu}

\author{Xavier Ros-Oton}
\address{The University of Texas at Austin, Department of Mathematics, 2515 Speedway, Austin, TX 78751, USA}
\email{ros.oton@math.utexas.edu}

\keywords{Obstacle problem; $p$-Laplacian; free boundary.}
\subjclass[2010]{35R35}

\thanks{AF and BK were supported by NSF-FRG grant  DMS-1361122. XR was supported by NSF grant DMS-1565186 and MINECO grant MTM2014-52402-C3-1-P (Spain)}

\maketitle

\begin{abstract}
We study the regularity of the free boundary in the obstacle for the $p$-Laplacian, $\min\bigl\{-\Delta_p u,\,u-\varphi\bigr\}=0$ in $\Omega\subset\R^n$.
Here, $\Delta_p u=\textrm{div}\bigl(|\nabla u|^{p-2}\nabla u\bigr)$, and $p\in(1,2)\cup(2,\infty)$.

Near those free boundary points where $\nabla \varphi\neq0$, the operator $\Delta_p$ is uniformly elliptic and smooth, and hence the free boundary is well understood.
However, when $\nabla \varphi=0$ then $\Delta_p$ is singular or degenerate, and nothing was known about the regularity of the free boundary at those points.

Here we study the regularity of the free boundary where $\nabla \varphi=0$.
On the one hand, for every $p\neq2$ we construct explicit global $2$-homogeneous solutions to the $p$-Laplacian obstacle problem whose free boundaries have a corner at the origin.
In particular, we show that the free boundary is in general not $C^1$ at points where $\nabla \varphi=0$.
On the other hand, under the ``concavity'' assumption $|\nabla \varphi|^{2-p}\Delta_p \varphi<0$, we show the free boundary is countably $(n-1)$-rectifiable and we prove a nondegeneracy property for $u$ at all free boundary points.
\end{abstract}

\vspace{4mm}

\section{Introduction}

In this paper we study the obstacle problem
\begin{equation}\label{obst-pb}
\min\bigl\{-\Delta_p u,\,u-\varphi\bigr\}=0\qquad\textrm{in}\quad \Omega\subset\R^n
\end{equation}
for the $p$-Laplacian operator
\[\Delta_p u=\textrm{div}\bigl(|\nabla u|^{p-2}\nabla u\bigr),\qquad 1<p<\infty.\]
The problem appears for example when considering minimizers of the constrained $p$-Dirichlet energy
\[\inf \left\{ \int_\Omega |\nabla v|^p \,:\, v\in W^{1,p}(\Omega), \hspace{3mm} v\geq\varphi \ \textrm{in}\ \Omega, \hspace{3mm} v=g\ \textrm{on}\ \partial\Omega \right\},\]
where $\varphi$ and $g$ are given smooth functions and $\Omega$ is a bounded smooth domain.

The regularity of solutions to \eqref{obst-pb} was recently studied by Andersson, Lindgren, and Shahgholian in~\cite{ALS15}.
Their main result establishes that if $\varphi\in C^{1,1}$ then
\[\sup_{B_r(x_0)}(u-\varphi)\leq Cr^2\qquad\textrm{for all}\quad r\in(0,1)\]
at any free boundary point $x_0\in \partial\{u>\varphi\}$.
Thus, solutions $u$ leave the obstacle $\varphi$ in a $C^{1,1}$ fashion at free boundary points $x_0$.

Notice that, near any free boundary point $x_0\in \partial\{u>\varphi\}$ at which $\nabla\varphi(x_0)\neq0$, the solution $u$ will satisfy $\nabla u\neq0$ as well and hence the operator $\Delta_p u$ is uniformly elliptic in a neighborhood of $x_0$.
Therefore, by classical results~\cite{Caf77,Caf98,PSU12}, the solution $u$ is $C^{1,1}$ near $x_0$, and the structure and regularity of the free boundary is well understood.

Thus, the main challenge in problem \eqref{obst-pb} is to understand the regularity of solutions and free boundaries near those free boundary points $x_0\in\partial\{u>\varphi\}$ at which $\nabla \varphi(x_0)=0$.
Our first main result is the following.

\begin{thm}\label{counterexample}
Let $p\in(1,2)\cup(2,\infty)$, and let $\varphi(x)=-|x|^2$ in $\R^2$.
There exists a $2$-homogeneous function $u:\R^2\to \R$ satisfying \eqref{obst-pb} in all of $\R^2$, and such that the set $\{u>\varphi\}$ is a cone with angle
\begin{equation*}
	\theta_0 = 2\pi \left( 1 - \sqrt{\frac{p-1}{2p}} \right) \neq \pi .
\end{equation*}
In particular, the free boundary has a corner at the origin.
\end{thm}

\begin{rem}
Let $u$ be a solution to \eqref{obst-pb} with $\varphi(x) = -|x|^2$ as in Theorem \ref{counterexample}.  For each $a \in \R$ and $b > 0$, $a - b u$ is a solution to \eqref{obst-pb} in $\R^2$ with $\varphi(x) = a - b |x|^2$ for which the contact set is a cone with angle $\theta_0\neq\pi$.
\end{rem}

\begin{rem}
Notice that for $p\in(2,\infty)$, $\theta_0 > \pi$ and thus $u$ is not convex. This is in contrast with the classical result of Caffarelli on the classifications of global solutions to the obstable problem for the Laplacian \cite{Caf98}.
\end{rem}

\begin{rem}%\label{counterexample_rmk}
In the process of constructing the solutions $u$ of Theorem \ref{counterexample}, for $p = 9$ we will construct a global solution $u$ to $\Delta_p u = 0$ in all of $\R^2$.
\end{rem}

In view of the above result, no $C^1$ regularity can be expected for the free boundary at points at which $\nabla\varphi=0$. Also, the lack of convexity of possible blow-up profiles
seems to be a major obstacle of understanding the fine structure of the free boundary at these points.

Still, an interesting question is to decide whether the free boundary has finite $\mathcal H^{n-1}$ measure near points at which the gradient of the obstacle vanishes.
A standard first step in this direction is to prove a nondegeneracy result stating that $u-\varphi$ cannot decay faster than quadratic at free boundary points.
\cite{ALS15} previously proved a similar nondegeneracy result at the free boundary points under the assumptions that $\varphi\in C^2$, $p>2$, and $\Delta_p\varphi\leq -c_0<0$.
However, if $\varphi\in C^2$ satisfies $\nabla \varphi(x_0)=0$ then $\Delta_p\varphi(x_0)=0$, and thus the result in~\cite{ALS15} can not be applied to free boundary points on $\{\nabla\varphi=0\}$.
We show the following. 

\begin{thm}\label{nondegeneracy}
Let $p\in(1,\infty)$, $\varphi\in C^2(B_1)$, and $u$ be a solution of \eqref{obst-pb} in $B_1$.
Assume that $\varphi$ satisfies 
\begin{equation}\label{concavity}
	|\nabla\varphi|^{2-p}{\rm div}\bigl(|\nabla \varphi|^{p-2}\nabla \varphi\bigr)\leq -c_0<0\qquad \textrm{in}\quad \{\nabla \varphi\neq0\}.
\end{equation}
Then for any free boundary point $x_0\in \partial\{u>\varphi\}\cap B_{1/2}$ there exists $c_1 > 0$ such that
\[\sup_{B_r(x_0)}(u-\varphi)\geq c_1 r^2\qquad \textrm{for}\quad r\in(0,1/2),\]
where the constant $c_1$ depends only on the modulus of continuity of $D^2\varphi$ and on the constant $c_0$ in \eqref{concavity}.
\end{thm}

\begin{rem}
The hypothesis \eqref{concavity} is nontrivial in the sense that \eqref{concavity} implies that either $\varphi$ is identically constant on $B_1$ or $\{\nabla \varphi\neq0\}$ is an open dense subset of $B_1$, see Lemma \ref{rectifiability_lemma} below.  In the case that $\varphi$ is identically constant on $B_1$, by the Hopf boundary point lemma~\cite[Theorem 5]{Vaz84} either $u \equiv \varphi$ in $B_1$ or $\Delta_p u = 0$ and $u > \varphi$ in $B_1$ and in particular the free boundary is an empty set.
\end{rem}

As a consequence of Theorem \ref{nondegeneracy} we can deduce that, under the hypotheses of Theorem \ref{nondegeneracy}, the free boundary is porous: i.e., there exists a $\delta > 0$ such that for every $B_r(x_0) \subseteq B_1$, there exists $B_{\delta r}(x) \subset B_r(x_0) \setminus \partial\{u>\varphi\}$.
The proof is standard and follows from combining the optimal regularity of solutions in~\cite{ALS15} with Theorem \ref{nondegeneracy} above.
Porosity of the free boundary implies that the free boundary has zero Lebesgue measure.  We in fact prove the stronger result that, under the hypotheses of Theorem \ref{nondegeneracy}, the free boundary $\partial \{u>\varphi\}$ is an $(n-1)$-dimensional rectifiable set. 

\begin{defi}
Let $0 \leq k \leq n$ be an integer.  We say a set $S \subseteq \R^n$ is countably $k$-rectifiable if there exists a set $E_0 \subset \R^n$ with $\mathcal{H}^k(E_0) = 0$ and a countable collection of Lipschitz maps $f_j : \R^k \rightarrow \R^n$ such that
\begin{equation*}
	S \subseteq E_0 \cup \bigcup_{j=1}^{\infty} f_j(\R^k).
\end{equation*}
\end{defi}

\begin{thm}\label{rectifiability}
Let $p\in(1,\infty)$, $\varphi\in C^2(B_1)$, and $u$ be a solution of \eqref{obst-pb} in $B_1$.
Assume that $\varphi$ satisfies \eqref{concavity}.  
Then, the free boundary $\partial \{u>\varphi\}$ is countably $(n-1)$-rectifiable.
\end{thm}

Related obstacle-type problems for the $p$-Laplacian have been studied in~\cite{KKPS00,LS03,CLRT14,CLR12}.
In those works, however, they studied the different problem
\begin{equation}\label{obst-pb2}
\Delta_p u=f(x)\chi_{\{u>0\}}\qquad \textrm{in}\quad \Omega\subset\R^n.
\end{equation}
It is important to notice that, when $p\neq2$, the obstacle problems \eqref{obst-pb} and \eqref{obst-pb2} are of quite different nature.
For example, when $f\equiv1$ solutions to \eqref{obst-pb2} are \emph{not} $C^{1,1}$ but $C^{1,\frac{1}{p-1}}$ near all free boundary points.

The paper is organized as follows.
In Section \ref{sec:counterexample_sec} we prove Theorem \ref{counterexample}.
Then, in Section \ref{sec:structfreebdry_sec} we prove Theorems \ref{nondegeneracy} and \ref{rectifiability}.

\section{Homogeneous degree-two solutions}
\label{sec:counterexample_sec}

We construct here the homogeneous solutions of Theorem \ref{counterexample}.

\begin{proof}[Proof of Theorem \ref{counterexample}]
Let $1 < p < \infty$ and $p \neq 2$, and let $\varphi(x)=-|x|^2$ in $\R^2$.
We will show that there exists a global solution $u(x_1,x_2)$ to \eqref{obst-pb} which is homogeneous of degree 2 and such that the free boundary consists of two rays meeting at an angle $\theta_0\neq \pi$.

We use polar coordinates $r e^{i\theta}$ on $\R^2$ where $r > 0$ and $\theta \in [0,2\pi]$.
We want to construct $u \in C^1(\R^2)$ such that
\begin{gather}
	\Delta_p u(re^{i\theta}) = 0, \quad u(re^{i\theta}) \geq -r^2 \quad\text{ for } \theta \in (0,\theta_0), \nonumber \\
	u(re^{i\theta}) = -r^2 \quad\text{ for } \theta \in [\theta_0,2\pi], \label{pLapHomog_prob1} \\
	%u(r) = u(re^{i\theta_0}) = -r^2, \quad
	 \frac{\partial u}{\partial\theta}(r) = \frac{\partial u}{\partial\theta} (re^{i\theta_0}) = 0. \nonumber
\end{gather}
%where $\theta_0 \in (0,2\pi)$. $\theta_0$ will be the angle between the two rays forming the free boundary and we will be particularly interested in the case $\theta_0 \neq \pi$.
Assume that $u(re^{i\theta}) = r^2 v(\theta)$ for some $2\pi$-periodic function $v \in C^1(\R) \cap C^{\infty}([0,\theta_0]) \cap C^{\infty}([\theta_0,2\pi])$.  We want to express $\Delta_p u(re^{i\theta}) = 0$ for $\theta \in (0,\theta_0)$ as a ordinary differential equation of $v$.  We compute
\begin{equation*}
	\nabla u = 2 \, r \, v(\theta) \, \frac{\partial}{\partial r} + v'(\theta) \, \frac{\partial}{\partial \theta}
\end{equation*}
and thus
\begin{align*}
	\Delta_p u &= \frac{1}{r} \, \frac{\partial}{\partial r} \left( r \, |\nabla u|^{p-2} \frac{\partial u}{\partial r} \right)
		+ \frac{1}{r^2} \, \frac{\partial}{\partial \theta} \left( |\nabla u|^{p-2} \frac{\partial u}{\partial \theta} \right) \\
	&= \frac{1}{r} \, \frac{\partial}{\partial r} \left( r^p \, (4 \, v^2 + (v')^2)^{(p-2)/2} \cdot 2 v \right)
		+ r^{p-2} \, \frac{\partial}{\partial \theta} \left( (4 \, v^2 + (v')^2)^{(p-2)/2} \, v' \right) \\
	&= r^{p-2} \, (4 \, v^2 + (v')^2)^{(p-2)/2} \left( 2p \, v + v''
		+ (p-2) \, \frac{4 \, v \, (v')^2 +(v')^2 \, v''}{4 \, v^2 + (v')^2} \right) .
\end{align*}
Thus we can rewrite $\Delta_p u = 0$ as
\begin{equation*}
	2p \, v + v'' + (p-2) \, \frac{4 \, v \, (v')^2 +(v')^2 \, v''}{4 \, v^2 + (v')^2} = 0 \,.
\end{equation*}
Solving for $v''$,
\begin{equation} \label{pLapHomog_ode}
	v'' = -v \frac{8p \, v^2 + (6p-8) \, (v')^2}{4 \, v^2 + (p-1) \, (v')^2} \,.
\end{equation}
Notice that \eqref{pLapHomog_prob1} is equivalent to $v$ satisfying \eqref{pLapHomog_ode} for $\theta \in (0,\theta_0)$ and
\begin{gather}
	v(\theta) \geq -1\quad \text{ for } \theta \in (0,\theta_0), \nonumber \\
	v(\theta) = -1 \quad\text{ for } \theta \in [\theta_0,2\pi], \label{pLapHomog_bc1} \\
	%v(0) = v(\theta_0) = -1, \quad
	v'(0) = v'(\theta_0) = 0. \nonumber
\end{gather}
Moreover, by integration by parts, \eqref{pLapHomog_ode}, and the homogeneity of $u$, for all $\zeta \in C^1_c(\R^2 \setminus \{0\})$ it holds 
\begin{align*}
	&-\int_{B_1} |\nabla u|^{p-2} \,\nabla u \cdot \nabla \zeta 
	\\&\hspace{7mm} = -\int_0^{\theta_0} \int_0^{\infty} |\nabla u|^{p-2} \,\nabla u \cdot \nabla \zeta \,r\,dr\,d\theta 
		- \int_{\theta_0}^{2\pi} \int_0^{\infty} |\nabla u|^{p-2} \,\nabla u \cdot \nabla \zeta \,r\,dr\,d\theta 
	\\&\hspace{7mm} = \int_0^{\theta_0} \int_0^{\infty} \Delta_p u \, \zeta \,r\,dr\,d\theta 
		+ \int_{\theta_0}^{2\pi} \int_0^{\infty} \Delta_p u \, \zeta \,r\,dr\,d\theta 
		+ \lim_{r \downarrow 0} \int_0^{2\pi} |\nabla u|^{p-2} \, D_r u \, \zeta \,r\,d\theta 
	\\&\hspace{7mm} = \int_{\theta_0}^{2\pi} \int_0^{\infty} \Delta_p (-r^2) \, \zeta \,r\,dr\,d\theta .
\end{align*}
Hence, $\Delta_p u = \Delta_p (-r^2)\,\chi_{\{\theta_0<\theta<2\pi\}} \leq 0$ weakly in $\R^2 \setminus \{0\}$.  This together with \eqref{pLapHomog_prob1} implies that $u$ is a solution to \eqref{obst-pb}.

Now let us solve \eqref{pLapHomog_ode}.  Set
\[X = v(\theta)\qquad\textrm{and}\qquad Y = v'(\theta)\]
so that we transform \eqref{pLapHomog_ode} into the first order system
\begin{equation} \label{pLapHomog_ode2}
	\left\{
	\begin{split}
	X' &= Y,  \\
	Y' &= -X \, \frac{8p \, X^2 + (6p-8) \, Y^2}{4 \, X^2 + (p-1) \, Y^2}\qquad \text{ on } (0,\theta_0) .
	\end{split}
	\right.
\end{equation}
Now, \eqref{pLapHomog_bc1} implies that
\begin{gather}
	X(\theta) \geq -1 \quad \text{ for } \theta \in (0,\theta_0), \nonumber \\
	(X(0),Y(0)) = (X(\theta_0),Y(\theta_0)) = (-1,0) \,. \label{pLapHomog_bc3}
\end{gather}
Notice that \eqref{pLapHomog_ode2} states that $X' = Y$ and $Y'$ equals a homogeneous degree one function of $(X,Y)$.  Thus it is convenient to set
\[X = \rho(\theta) \, \cos(\psi(\theta))\qquad \textrm{and}\qquad Y = \rho(\theta) \, \sin(\psi(\theta))\]
for some functions $\rho$ and $\psi$,
so that \eqref{pLapHomog_ode2} is equivalent to
$$
\left\{
\begin{split}
	\frac{\rho'}{\rho} \, \cos(\psi) - \psi' \, \sin(\psi) &= \sin(\psi),  \\
	\frac{\rho'}{\rho} \, \sin(\psi) + \psi' \, \cos(\psi) &= -\cos(\psi) \, \frac{8p \, \cos^2(\psi) + (6p-8) \, \sin^2(\psi)}{4 \, \cos^2(\psi) + (p-1) \, \sin^2(\psi)} \,
	\end{split}
	\right.
$$
for all $\theta \in (0,\theta_0)$.  Let
\begin{equation*}
	F_p(\psi) := \frac{8p \, \cos^2(\psi) + (6p-8) \, \sin^2(\psi)}{4 \, \cos^2(\psi) + (p-1) \, \sin^2(\psi)} - 1
	= \frac{(8p-4) \, \cos^2(\psi) + (5p-7) \, \sin^2(\psi)}{4 \, \cos^2(\psi) + (p-1) \, \sin^2(\psi)}
\end{equation*}
so that
\begin{equation} \label{pLapHomog_ode3}
	\left\{
	\begin{split}
	\frac{\rho'}{\rho} \, \cos(\psi) - \psi' \, \sin(\psi) &= \sin(\psi),  \\
	\frac{\rho'}{\rho} \, \sin(\psi) + \psi' \, \cos(\psi) &= -\cos(\psi) - \cos(\psi) \, F(\psi)
	\end{split}
	\right.
\end{equation}
for all $\theta \in (0,\theta_0)$.  Note that \eqref{pLapHomog_ode3} can be rewritten as
\begin{align}
	\frac{\rho'}{\rho} &= -\cos(\psi) \, \sin(\psi) \, F_p(\psi), \label{pLapHomog_ode4} \\
	\psi' &= -1 - \cos^2(\psi) \, F_p(\psi) , \label{pLapHomog_ode5}
\end{align}
and this system can be solved by first solving \eqref{pLapHomog_ode5} to find $\psi$, and then integrating \eqref{pLapHomog_ode4} to find $\rho$.  We compute that
\begin{align*}
	\frac{\partial}{\partial p} F_p(\psi)
	&= \frac{\partial}{\partial p} \left( \frac{(8p-4) \, \cos^2(\psi) + (5p-7) \, \sin^2(\psi)}{4 \, \cos^2(\psi) + (p-1) \, \sin^2(\psi)} \right)
	\\&= \frac{\partial}{\partial p} \left( \frac{(3p+3) \, \cos^2(\psi) + 5p-7}{(-p+5) \, \cos^2(\psi) + p-1} \right)
	\\&= \frac{(3 \, \cos^2(\psi) + 5) \, (5 \, \cos^2(\psi) - 1) - (3 \, \cos^2(\psi) - 7) \, (-\cos^2(\psi) + 1)}{((-p+5) \, \cos^2(\psi) + p-1)^2}
	\\&= \frac{18 \, \cos^4(\psi) + 12 \, \cos^2(\psi) + 2}{((-p+5) \, \cos^2(\psi) + p-1)^2} > 0
\end{align*}
for all $\psi \in [0,2\pi]$, so
\begin{equation*}
	1+ \cos^2(\psi) \, F_p(\psi) \geq 1 + \cos^2(\psi) \, F_1(\psi) = 1 + \cos^2(\psi) - \frac{1}{2} \, \sin^2(\psi) = \frac{1}{2} + \frac{3}{2} \, \cos^2(\psi) > 0
\end{equation*}
for all $\psi \in [0,2\pi] \setminus \{\pi/2,3\pi/2\}$. Note that, when $\psi = \pi/2, 3\pi/2$, $F_p(\psi)$ degenerates as $p \downarrow 1$, but $1+ \cos^2(\psi) \, F_p(\psi) = 1$ for all $p > 1$.  Thus, solving \eqref{pLapHomog_ode5}, we find that
$\psi(\theta) = \Theta^{-1}(\theta)$ for all $\theta \in [0,\theta_0]$ where $\Theta : \R \rightarrow \R$ is the strictly decreasing function defined by
\begin{equation} \label{pLapHomog_psi}
	\Theta(\psi) := -\int_{\psi(0)}^{\psi} \frac{d\sigma}{1 + \cos^2(\sigma) \, F_p(\sigma)}
\end{equation}
for $\psi(0)$ to be determined.
Integrating \eqref{pLapHomog_ode4} over $[0,\theta]$, we obtain
\begin{equation} \label{pLapHomog_rho}
	\rho(\theta) = \rho(0) \, \exp \left( - \int_0^{\theta} \cos(\psi(\tau)) \, \sin(\psi(\tau)) \, F_p(\psi(\tau)) \, d\tau \right)
\end{equation}
for all $\theta \in [0,\theta_0]$ and for $\rho(0)$ to be determined.  Notice that $\psi, \rho \in C^{\infty}([0,\theta_0])$ and thus $v = \rho \, \cos(\psi) \in C^{\infty}([0,\theta_0])$. 

It remains to determine $\theta_0$, $\psi(0)$, and $\rho(0)$, and to verify that \eqref{pLapHomog_bc3} holds true. 

To this aim, we observe that \eqref{pLapHomog_bc3} is equivalent to
\begin{equation} \label{pLapHomog_bc5}
	\rho(0) = \rho(\theta_0) = 1, \quad \psi(0) = \pi, \quad \psi(\theta_0) = -(2k-1) \pi
\end{equation}
for some integer $k \geq 1$. Then, in view of the fact that $X' = Y$ (and so $X(\theta)$ attains its minimum value when $\psi(\theta) = -(2j-1) \pi$ for some integer $j$), we see that
\begin{equation} \label{pLapHomog_bc6}
	\rho(\theta) \leq 1 \text{ whenever } \psi(\theta) = -(2j-1) \pi,  \text{ for } j = 1,2,\ldots,k-1.
\end{equation}
Hence, by \eqref{pLapHomog_bc5}, we should choose $\rho(0) = 1$ and $\psi(0) = \pi$.  To choose $\theta_0$ observe that, by \eqref{pLapHomog_psi}, $\psi(\theta_0) = -(2k-1) \, \pi$ if and only if
\begin{equation*}
	\theta_0 = \int_{-(2k-1) \, \pi}^{\pi} \frac{d\sigma}{1 + \cos^2(\sigma) \, F_p(\sigma)} = 4k \int_0^{\pi/2} \frac{d\sigma}{1 + \cos^2(\sigma) \, F_p(\sigma)},
\end{equation*}
where the last step follows by symmetry.  We compute that
\begin{align*} %\label{theta0_integral}
	&\int_0^{\pi/2} \frac{d\sigma}{1 + \cos^2(\sigma) \, F_p(\sigma)} \nonumber
	\\&\hspace{10mm} = \int_0^{\pi/2} \frac{4 \, \cos^2(\sigma) + (p-1) \, \sin^2(\sigma)}{4 \, \cos^2(\sigma) + (p-1) \, \sin^2(\sigma) + (8p-4) \, \cos^4(\sigma)
		+ (5p-7) \, \cos^2(\sigma) \, \sin^2(\sigma)} \, d\sigma \nonumber
	\\&\hspace{10mm} = \int_0^{\pi/2} \frac{(-p+5) \, \cos^2(\sigma) + p-1}{(3p+3) \, \cos^4(\sigma) + (4p-2) \, \cos^2(\sigma) + p-1} \, d\sigma \nonumber
	\\&\hspace{10mm} = \int_0^{\pi/2} \frac{(-p+5) \, \cos^2(\sigma) + p-1}{(3 \, \cos^2(\sigma) + 1) \, ((p+1) \, \cos^2(\sigma) + p-1)} \, d\sigma \nonumber
	\\&\hspace{10mm} = \int_0^{\pi/2} \left( \frac{2}{3 \, \cos^2(\sigma) + 1} - \frac{p-1}{(p+1) \, \cos^2(\sigma) + p-1} \right) \, d\sigma \nonumber
	\\&\hspace{10mm} = \int_0^{\pi/2} \left( \frac{2}{\tan^2(\sigma) + 4} - \frac{p-1}{(p-1) \, \tan^2(\sigma) + 2p} \right) \, \sec^2(\sigma) \, d\sigma \nonumber
	\\&\hspace{10mm} = \int_0^{\infty} \left( \frac{2}{t^2 + 4} - \frac{p-1}{(p-1) \, t^2 + 2p} \right) \, dt \nonumber
	\\&\hspace{10mm} = \left[ \arctan \left(\frac{t}{2}\right) - \sqrt{\frac{p-1}{2p}} \, \arctan\left(\sqrt{\frac{p-1}{2p}} \, t \right) \right]_0^{\infty} \nonumber
	\\&\hspace{10mm} = \frac{\pi}{2} \left( 1 - \sqrt{\frac{p-1}{2p}} \right) ,
\end{align*}
where we let $t = \tan(\sigma)$.  Thus, we need to choose
\begin{equation} \label{theta0}
	\theta_0 = 2k\pi \left( 1 - \sqrt{\frac{p-1}{2p}} \right)
\end{equation}
for some integer $k \geq 1$.  Since
\[\frac{k\pi}{2} < 2k\pi \left( 1 - \sqrt{\frac{p-1}{2p}} \right) = \theta_0 < 2\pi ,\]
we deduce that $k \leq 3$. 

Notice that, for each $k = 1,2,3$, $\theta_0$ given by \eqref{theta0} is decreasing as a function of $p$.
In particular, when $k = 1$, $\theta_0 = 2\pi$ for $p = 1$, $\pi < \theta_0 < 2\pi$ for $1 < p < 2$, $\theta_0 = \pi$ for $p = 2$, and $0 < \theta_0 < \pi$ for $2 < p < \infty$, see Figure \ref{fig-p}.

\begin{figure}
\begin{center}
\includegraphics{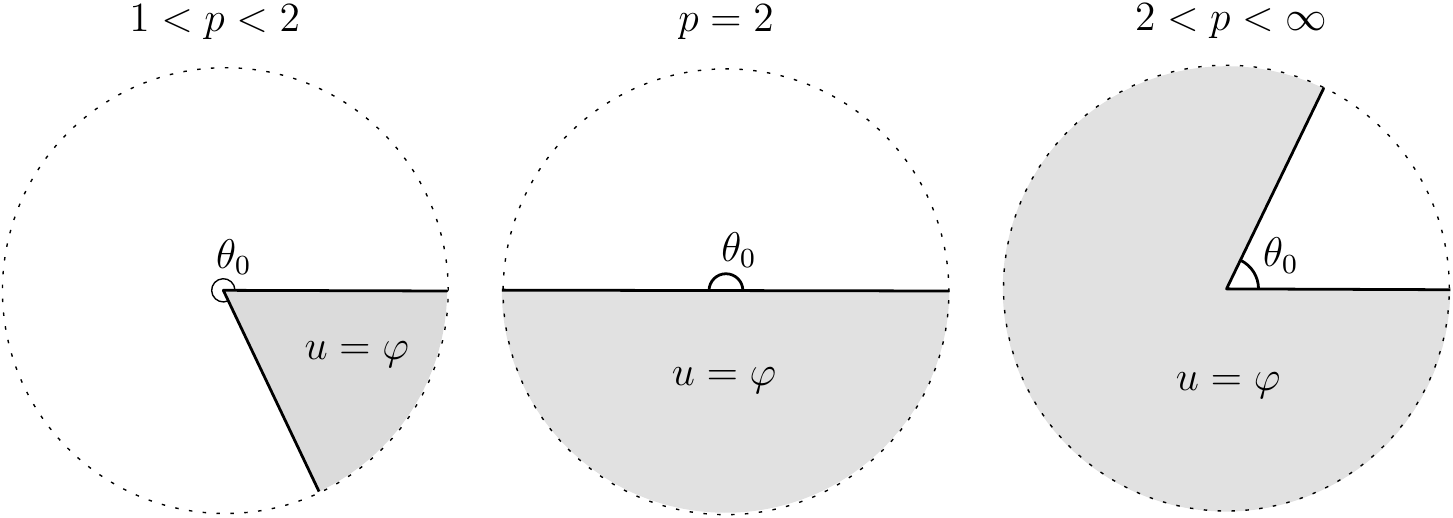}
\end{center}
\caption{The angle $\theta_0$ and the contact set $\{u=\varphi\}$ of the homogeneous solution for $1<p<2$, $p=2$, and $2<p<\infty$, respectively.}
\label{fig-p}
\end{figure}

Hence if $p \in (1,2) \cup (2,\infty]$, by setting $k = 1$ we can construct $u$ for which
\begin{equation*}
	\theta_0 = 2\pi \left( 1 - \sqrt{\frac{p-1}{2p}} \right) \in (0,\pi) \cup (\pi,2\pi) .
\end{equation*}

When $k = 2$, for all $1 < p < 2$ we have $\theta_0 > 2\pi$ and consequently we do not obtain a solution $u$, for $p = 2$ we have $\theta_0 = 2\pi$, and for all $2 < p < \infty$ we obtain a solution $u$ with $\pi < \theta_0 < 2\pi$.  Similarly, when $k = 3$, we do not obtain a solution $u$ for $1 < p < 9$, $\theta_0 = 2\pi$ for $p = 9$, and we obtain a solution $u$ with $\pi < \theta_0 < 2\pi$ for $9 < p < \infty$.

To conclude, we need to verify $\rho(\theta_0) = 1$. For this, suppose that $\theta$ is such that $\psi(\theta) = -(2j-1) \pi$ for an integer $j \geq 1$.  Observe that by \eqref{pLapHomog_psi} and symmetry,
\begin{equation*}
	\int_0^{\theta} \cos(\psi(\tau)) \, \sin(\psi(\tau)) \, F_p(\psi(\tau)) \, d\tau
	= \int_{-(2j-1) \pi}^{\pi} \frac{\cos(\sigma) \, \sin(\sigma) \, F_p(\sigma)}{1 + \cos^2(\sigma) \, F_p(\sigma)} \, d\sigma = 0 \,,
\end{equation*}
where $\sigma = \psi(\tau)$.  Therefore by \eqref{pLapHomog_rho} $\rho(\theta) = \rho(0) = 1$.  In particular, when $j = k$, we get $\rho(\theta_0) = \rho(0)=1$. 

Notice that for $k = 1$ the contact set $\{u>\varphi\}$ is precisely $\{\theta_0\leq\theta\leq 2\pi\}$, whereas for $k = 2,3$ the contact set $\{u>\varphi\}$ is the union of $\{\theta_0\leq\theta\leq 2\pi\}$ and the rays $\psi(\theta) = -(2j-1)\pi/2$, i.e. $\theta = 2\pi j \left( 1 - \sqrt{\frac{p-1}{2p}} \right)$, for $j = 1,\ldots,k-1$.
\end{proof}

\begin{rem}
Observe that when $k = 1$ and $p = 2$, the above argument produces a solution to $u$ to \eqref{pLapHomog_prob1} with $\theta_0 = \pi$.  In other words, the contact set $\{u = \varphi\}$ is a half-space.
On the other hand, 
when $k = p = 2$, or when $k = 3$ and $p = 9$, the above argument produces solutions $\rho$ and $\psi$ to \eqref{pLapHomog_ode3} with $\theta_0 = 2\pi$ so that
\begin{equation*}
	\rho(0) = \rho(2\pi), \quad \psi(2\pi) - \psi(0) = -2k\pi,
\end{equation*}
and we thereby obtain $u \in C^1(\R^2)$ such that $\Delta_p u = 0$ in all of $\R^2$.  Note that $\rho(0) > 0$ and $\psi(0)$ are arbitrary and this corresponds to the invariance of $\Delta_p u = 0$ in $\R^2$ under scaling and rotations.

While this solution for $p=9$ is new (at least to our knowledge), these solutions for $p = 2$ are well-known.  Indeed, when $p = 2$, \eqref{pLapHomog_ode} reduces to
\begin{equation*}
	v'' = -4 v,
\end{equation*}
which obviously has the solution
\begin{equation} \label{pLapHomog_soln_pis2}
	v(\theta) = A \cos(2\theta) + B \sin(2\theta)
\end{equation}
for constants $A,B \in \R$.  Assuming that $v$ is given by \eqref{pLapHomog_soln_pis2} for all $\theta \in [0,\theta_0]$ and $v$ satisfies the boundary conditions $v(0) = v(\theta_0) = -1$ and $v'(0) = v'(\theta_0) = 0$, we obtain $\theta_0 = \pi$, $A = -1$, and $B = 0$ so that
\begin{equation*}
	u(x) = r^2 v(re^{i\theta}) = -|x|^2 + 2 (x_2)_+^2
\end{equation*}
so that $w = u-\varphi$ is the well-known global solution $w = 2 (x_2)_+^2$ to the obstacle problem $\min\{\Delta w,w\} = 0$ in $\R^2$.  If instead we assume that $v$ is given by \eqref{pLapHomog_soln_pis2} for all $\theta \in [0,2\pi]$, then
\begin{equation*}
	u(x) = r^2 v(re^{i\theta}) = A (x_1^2 - x_2^2) + 2 B x_1 x_2,
\end{equation*}
giving us the usual homogeneous degree two harmonic polynomials.
\end{rem}

\section{Structure of the free boundary}
\label{sec:structfreebdry_sec}

In this section we prove Theorem \ref{nondegeneracy} and Theorem \ref{rectifiability}. 
First we will use the implicit function theorem to show that \eqref{concavity} implies that either $\varphi$ is a constant function or $\{\nabla \varphi = 0\}$ is countably $(n-1)$-rectifiable.
One immediate consequence is that $\{\nabla \varphi \neq 0\}$ is either empty or an open dense subset, which we use to prove Theorem \ref{nondegeneracy}.
Another immediate consequence is Theorem \ref{rectifiability}. 

\begin{lem} \label{rectifiability_lemma}
Let $p\in(1,\infty)$ and $\varphi \in C^2(B_1)$ such that \eqref{concavity} holds true.  Then either $\varphi$ is identically constant on $B_1$ or $\{\nabla \varphi = 0\}$ is countably $(n-1)$-rectifiable.
\end{lem}
\begin{proof}
First we will show that \eqref{concavity} implies that either $\varphi$ is identically constant on $B_1$ or
\begin{equation} \label{HessianLB}
	|D^2 \varphi| \geq \frac{c_0}{n+p-2} \quad \textrm{in} \quad B_1,
\end{equation}
where $|D^2\varphi(x)|$ denotes the operator norm of the matrix $D^2\varphi(x)$.

By \eqref{concavity},
\begin{equation*}
	c_0 \leq \left| \Delta \varphi + (p-2) \frac{\langle \nabla \varphi, D^2 \varphi \nabla \varphi \rangle}{|\nabla \varphi|^2} \right|
		\leq (n+p-2) \, |D^2 \varphi| \quad \textrm{in} \quad \{\nabla \varphi\neq0\}.
\end{equation*}
Hence, noting that $\varphi \in C^2(B_1)$, we can express $B_1$ as the union of the disjoint sets
\begin{equation*}
	\Big\{ |D^2 \varphi| \geq \frac{c_0}{n+p-2} \Big\} \quad \textrm{and} \quad \operatorname{int}\{ \nabla \varphi = 0 \} ,
\end{equation*}
which are both relatively open and closed in $B_1$, and use the connectedness of $B_1$ to reach our desired conclusion.

Now, suppose \eqref{HessianLB} holds true.  Let $x_0 \in B_1 \cap \{\nabla \varphi = 0\}$.  By \eqref{HessianLB}, $D^2 \varphi(x_0)$ has rank $k \geq 1$.  Hence after an orthogonal change of variables, we may assume that
\begin{equation*}
	D^2 \varphi(x_0) = \left( \begin{matrix} A & 0 \\ 0 & 0 \end{matrix} \right)
\end{equation*}
for some diagonal $k \times k$ matrix $A$ with full rank.  By the implicit function theorem, there is an open neighborhood of $x_0$ in which $M = \{ D_i \varphi = 0 \text{ for } i = 1,2,\ldots,k \}$ is a $C^1$ $(n-k)$-dimensional submanifold and $\{ \nabla \varphi = 0 \} \subseteq M$.  Therefore $\{ \nabla \varphi = 0 \}$ is countably $(n-1)$-rectifiable.
\end{proof}

Next we will prove Theorem \ref{nondegeneracy}.
For this, we will need the following Lemma.

\begin{lem}\label{lem1}
Let $\varphi\in C^2(B_1)$ be a function satisfying \eqref{concavity}.
Let $x_0\in B_{1/2}$ be such that $\nabla\varphi(x_0)=0$.
Then, there exists $\varepsilon>0$ and $\delta>0$ such that
\[\Delta_p\left(\varphi(x)+\varepsilon\frac{|x|^2}{2}\right)\leq 0\qquad \textrm{in}\quad B_\delta(x_0).\]
The constants $\varepsilon$ and $\delta$ depend only on the modulus of continuity of $D^2\varphi$ and on the constant $c_0$ in \eqref{concavity}.
\end{lem}

\begin{proof}
We may assume $x_0=0$.
Let us denote
\[\begin{split}
\widetilde \Delta_pw:=|\nabla w|^{2-p}\Delta_p w &=|\nabla w|^{2-p}\textrm{div}\bigl(|\nabla w|^{p-2}\nabla w\bigr)\\
&=\Delta w+(p-2)\frac{\langle\nabla w,D^2w\nabla w\rangle}{|\nabla w|^2}\\
&=\Delta w+(p-2)\Delta_\infty w
\end{split}\]
wherever $\nabla w \neq 0$.  We know that by \eqref{concavity}
\begin{equation}\label{lem1_concavity}
	\widetilde \Delta_p\varphi\leq -c_0<0 \qquad\textrm{in}\quad B_1\cap\{\nabla \varphi\neq0\},
\end{equation}
and we want to show that $\widetilde \Delta_p(\varphi+\frac12\varepsilon|x|^2)\leq 0$ in $B_\delta\cap\{\nabla \varphi + \varepsilon x \neq0\}$.

Let
\[\lambda_{\min}(x) = \lambda_1(x) \leq \lambda_2(x) \leq \cdots \leq \lambda_n(x) = \lambda_{\max}(x)\]
denote the eigenvalues of $D^2\varphi(x)$ and $\lambda_i=\lambda_i(0)$, $\lambda_{\min}=\lambda_{\min}(0)$, and $\lambda_{\max}=\lambda_{\max}(0)$.
By continuity of $D^2\varphi$, we have that $\lambda_i(x)$ are continuous in $x$.

\vspace{3mm}

\noindent \emph{Case 1.} Assume first that $\lambda_{\max}\leq 0$, i.e., $\lambda_i\leq0$ for all $i=1,...,n$.

Noting that 
\[\lambda_{\min}(x)\leq \Delta_\infty\varphi(x)\leq \lambda_{\max}(x)\]
and using \eqref{lem1_concavity}, we obtain for every $x \in B_{\delta} \cap \{\nabla \varphi\neq0\}$ that 
\[(n+p-2)\lambda_{\min}(x)\leq \bigl(\lambda_1(x)+...+\lambda_n(x)\bigr)+(p-2)\lambda_{\min}(x)\leq -\frac12c_0\qquad\textrm{if}\quad p>2\]
and
\[(n+p-2)\lambda_{\min}(x)\leq \bigl(\lambda_1(x)+...+\lambda_n(x)\bigr)+(p-2)\lambda_{\max}(x)\leq -\frac12c_0\qquad\textrm{if}\quad p<2,\]
provided that $\delta>0$ is small enough.
In any case, we find $\lambda_{\min}(x)\leq \frac{-1}{2(n+p-2)} c_0$ in $B_\delta$.
Moreover, if $\delta$ is small, then $\lambda_{\max}(x)\leq \varepsilon$ in $B_\delta$.
Hence, for all $x \in B_{\delta}$ such that $\nabla \varphi(x) \neq 0$ and $\nabla \varphi(x) + \epsilon x \neq 0$ we have
\begin{align}\label{lem1_case1_concl1}
	\widetilde \Delta_p\left(\varphi(x)+\frac12\varepsilon|x|^2\right) 
	&\leq \lambda_{\min}(x)+(n-1)\lambda_{\max}(x)+(p-2)\lambda_{\max}(x)+(n+p-2)\varepsilon\nonumber
	\\&\leq -\frac{1}{2(n+p-2)} c_0+(2n+2p-5)\varepsilon \leq 0\qquad\textrm{if}\quad p>2, 
\end{align}
provided $\varepsilon$ is sufficiently small.
Since $B_{\delta} \cap \{\nabla \varphi \neq 0\}$ is an open dense subset of $B_{\delta}$ (thanks to Lemma \ref{rectifiability_lemma}), we have \eqref{lem1_case1_concl1} for all $x \in B_{\delta}$ such that $\nabla \varphi(x) + \epsilon x \neq 0$.  Similarly, for all $x \in B_{\delta}$ such that $\nabla \varphi(x) + \epsilon x \neq 0$ we obtain 
\begin{align*}
	\widetilde \Delta_p\left(\varphi(x)+\frac12\varepsilon|x|^2\right) 
	&\leq \lambda_{\min}(x)+(n-1)\lambda_{\max}(x)+(p-2)\lambda_{\min}(x)+(n+p-2)\varepsilon
	\\&\leq -\frac{p-1}{2(n+p-2)} c_0+(2n+p-3)\varepsilon \leq 0\qquad\textrm{if}\quad p<2
\end{align*}
provided $\varepsilon$ is sufficiently small, as desired.

\vspace{3mm}

\noindent \emph{Case 2.} Let us assume now that $\lambda_{\max}>0$.

Since $\varphi\in C^2(B_1)$, there is a modulus of continuity $\omega$ such that
\begin{gather}
	\left|D^2\varphi(x)-D^2\varphi(0)\right|\leq \omega(|x|),\nonumber\\
	\left|\nabla\varphi(x)-D^2\varphi(0)x\right|\leq |x|\,\omega(|x|).\label{lem1_quadratic1}
\end{gather}
After an affine change of variables, we may assume that $D^2\varphi(0)$ is a diagonal matrix,
\begin{equation}\label{lem1_quadratic2}
	D^2\varphi(0)=\left(\begin{array}{ccc}
	\lambda_1 & \cdots & 0 \\
	\vdots & \ddots & \vdots \\
	0 & \cdots & \lambda_n
	\end{array}\right).
\end{equation}
Notice that by \eqref{lem1_quadratic1} and \eqref{lem1_quadratic2} we have
\begin{equation}\label{lem1_quadratic3}
	\bigl|\nabla \varphi(x)-(\lambda_1x_1,...,\lambda_nx_n)\bigr|\leq |x|\,\omega(|x|).
\end{equation}
By \eqref{lem1_quadratic1}, \eqref{lem1_quadratic2}, and \eqref{lem1_quadratic3}, for any $x\in B_1$ such that $\nabla\varphi(x)\neq0$ we have
\[\widetilde\Delta_p\varphi(x)=(\lambda_1+...+\lambda_n)+o(1)+(p-2)\frac{\sum_i \lambda_i^3 x_i^2+o(|x|^2)}{\sum_i\lambda_i^2x_i^2+o(|x|^2)}.\]
In particular, since $\{\nabla \varphi \neq 0\}$ is dense in $B_1$ (as a consequence of Lemma \ref{rectifiability_lemma}), for each $\lambda_i\neq0$ there is a sequence of points $x^{(k)}\to0$ such that
\[\widetilde \Delta_p\varphi(x^{(k)})\to (\lambda_1+...+\lambda_n)+(p-2)\lambda_i.\]
This together with \eqref{lem1_concavity} means that
\begin{align}\label{lem1_lambda}
	(\lambda_1+\cdots+\lambda_n)+(p-2)\lambda_{\max} &\leq -c_0\qquad \textrm{if}\quad p>2,\\
	(\lambda_1+\cdots+\lambda_n)+(p-2)\lambda_{\min} &\leq -c_0\qquad \textrm{if}\quad p<2.\nonumber
\end{align}

Let $\varepsilon>0$ to be chosen later, and let $\delta>0$ small so that $\omega(|x|)<\varepsilon^5$ for $|x|<\delta$.
Then for all $x \in B_{\delta}$ such that $\nabla \varphi(x) + \varepsilon x \neq 0$ we have 
\begin{align}\label{lem1_computation1}
	&\widetilde\Delta_p\left(\varphi(x)+\frac12\varepsilon|x|^2\right) \nonumber
	\\&\hspace{10mm} = \Delta\varphi(x)+n\varepsilon+(p-2)\frac{\langle\nabla \varphi(x)+\varepsilon x, 
		(D^2\varphi(x)+\varepsilon{\rm Id})(\nabla \varphi(x)+\varepsilon x)\rangle}{|\nabla\varphi+\varepsilon x|^2}. 
\end{align}
We must now be careful and choose $\varepsilon>0$ such that the denominator is not zero for any $|x|<\delta$.
For this, let $\theta>0$ to be chosen later, and $k\in\{1,...,n+1\}$ be such that no $\lambda_i$ satisfies $-(k+1)\theta<\lambda_i<-k\theta$.
Take $\varepsilon=(k+\frac12)\theta$, and notice that $\frac12\theta\leq \varepsilon\leq (n+2)\theta$ and
\begin{equation}\label{frt}
|\lambda_i+\varepsilon|\geq \frac{\varepsilon}{2(n+2)}\qquad\textrm{for all}\quad i=1,...,n.
\end{equation}
Suppose $p>2$.  By \eqref{lem1_quadratic1}, \eqref{lem1_quadratic2}, and \eqref{lem1_quadratic3}, we find that for $|x|<\delta$ 
\begin{align}\label{lem1_computation2}
\frac{\langle\nabla \varphi(x)+\varepsilon x,(D^2\varphi(x)+\varepsilon{\rm Id})(\nabla \varphi(x)+\varepsilon x)\rangle}{|\nabla\varphi+\varepsilon x|^2}&\leq \frac{\sum_i(\lambda_i+\varepsilon)^3x_i^2+\varepsilon^4|x|^2}{ \sum_i(\lambda_i+\varepsilon)^2x_i^2-\varepsilon^4|x|^2}\nonumber\\
&= \frac{\sum_i\left[(\lambda_i+\varepsilon)^3-\varepsilon^4\right]x_i^2}{ \sum_i\left[(\lambda_i+\varepsilon)^2-\varepsilon^4\right]x_i^2}\nonumber\\
&\leq {\max}_i\frac{(\lambda_i+\varepsilon)^3+\varepsilon^4}{(\lambda_i+\varepsilon)^2-\varepsilon^4}\nonumber\\
&= {\max}_i\left(\lambda_i+\varepsilon+\frac{(\lambda_i+\varepsilon)\varepsilon^4+\varepsilon^4}{ (\lambda_i+\varepsilon)^2-\varepsilon^4}\right)\nonumber\\
&\leq \lambda_{\max}+\varepsilon+C\varepsilon^2
\end{align}
provided that $\varepsilon>0$ is small enough, where in the last inequality we used \eqref{frt}.
Now using \eqref{lem1_quadratic1}, \eqref{lem1_quadratic2}, \eqref{lem1_computation2}, and \eqref{lem1_lambda}, it follows by\eqref{lem1_computation1} that
\[\widetilde \Delta_p\left(\varphi(x)+\frac12\varepsilon|x|^2\right)\leq (\lambda_1+...+\lambda_n) +(p-2)\lambda_{\max}+n\varepsilon+C\varepsilon^2\leq 0\]
on $B_{\delta} \setminus \{0\}$, provided that $\varepsilon>0$ is small enough.  (Recall from the argument above that $\nabla \varphi(x) + \varepsilon x = 0$ if and only if $x = 0$.)

On the other hand, if $p<2$, then the same argument yields
\[\widetilde \Delta_p\left(\varphi(x)+\frac12\varepsilon|x|^2\right)\leq (\lambda_1+...+\lambda_n) +(p-2)\lambda_{\min}+n\varepsilon+C\varepsilon^2\leq 0\]
on $B_{\delta} \setminus \{0\}$, and thus we are done.
\end{proof}

Using the previous Lemma, we can now establish the following nondegeneracy property.

\begin{proof}[Proof of Theorem \ref{nondegeneracy}]
We claim that for every free boundary point $y_0 \in B_{1/2} \cap \partial\{u>\varphi\}$ there exists $\delta=\delta(y_0) > 0$ and $c=c(y_0) > 0$ such that
\begin{equation} \label{nondegeneracy_eqn}
	\sup_{B_r(y)} (u-\varphi) \geq c r^2 \qquad \textrm{for}\quad r \in (0,\delta), \, y \in B_{\delta}(y_0) \cap \overline{\{u>\varphi\}}.
\end{equation}
The conclusion of Theorem \ref{nondegeneracy} then follows from a standard covering argument.

In the case where $\nabla \varphi(y_0) \neq 0$, we may choose $\delta$ so that $\nabla \varphi \neq 0$ in $B_{4\delta}(y_0)$. In this way $\Delta_p u$ is uniformly elliptic in $B_{4\delta}(y_0)$ and \eqref{nondegeneracy_eqn} follows by the classical theory (see for instance~\cite[Lemma 5]{Caf98}).

Suppose $\nabla \varphi(y_0) = 0$.  By Lemma \ref{lem1}, there are $\varepsilon>0$ and $\delta>0$ such that
\[v(x)=\varphi(x)+\varepsilon\frac{|x-y|^2}{2}\]
satisfies $\Delta_p v\leq 0$ in $B_{2\delta}(y_0)$.
By continuity, we may assume $y \in B_{\delta}(y_0) \cap \{u>\varphi\}$.
Then, for any $r<\delta$, we have $\Delta_p u \geq \Delta_p v$ in $\{u>\varphi\} \cap B_r(y)$.
Moreover, $u(y)\geq \varphi(y)=v(y)$.
It follows from the comparison principle that there is $z_y\in \partial(\{u>\varphi\}\cap B_r(y))$ such that $u(z_y)\geq v(z_y)$.
Since $u<v$ on $\{u=\varphi\}$ it follows that $z_y\in \{u>\varphi\}\cap \partial B_r(x_0)$, and so
\[u(z_y)-\varphi(z_y) = u(z_y)-v(z_y)+\frac{\varepsilon r^2}{2} \geq \frac{\varepsilon r^2}{2}. \qedhere \]
\end{proof}

As a direct consequence of Lemma \ref{rectifiability_lemma} and the classical theory of the obstacle problem for uniformly elliptic operators, we obtain Theorem \ref{rectifiability}.

\begin{proof}[Proof of Theorem \ref{rectifiability}]
Let us express the free boundary $\Gamma = \partial \{u > \varphi\}$ as
\begin{equation} \label{rectif_decomp}
	\Gamma = \Gamma_1 \cup \Gamma_2 \quad\textrm{where}\quad
	\Gamma_1 = \Gamma \cap \{\nabla \varphi \neq 0\} \quad\textrm{and}\quad
	\Gamma_2 = \Gamma \cap \{\nabla \varphi = 0\}.
\end{equation}
In order to show that the free boundary $\Gamma$ is countably $(n-1)$-rectifiable, it suffices to show that each of the sets $\Gamma_1$ and $\Gamma_2$ are countably $(n-1)$-rectifiable.  For every $x_0 \in \Gamma_1$ there exists a $\delta > 0$ such that $\nabla \varphi \neq 0$ in $B_{\delta}(x_0)$ and thus $\Delta_p u$ is uniformly elliptic in $B_{\delta}(x_0)$.  Hence $\Gamma_1 \cap B_{\delta/2}(x_0) = \Gamma \cap B_{\delta/2}(x_0)$ is a countably $(n-1)$-rectifiable set with finite $(n-1)$-dimensional measure (see for instance~\cite[Corollary 4]{Caf98}).  It follows from a covering argument that $\Gamma_1$ is countably $(n-1)$-rectifiable.  By Lemma \ref{rectifiability_lemma}, $\{\nabla \varphi = 0\}$ is countably $(n-1)$-rectifiable and thus $\Gamma_2$ is countably $(n-1)$-rectifiable.
\end{proof}

\begin{rem}
Let $\Gamma = \Gamma_1 \cup \Gamma_2$ be as in \eqref{rectif_decomp}.
Our argument show that, for each $i = 1,2$ and $x_0 \in \Gamma_i$, there exists a $\delta > 0$ such that $\Gamma_i \cap B_{\delta}(x_0)$ is a relatively closed, countably $(n-1)$-rectifiable subset of $B_{\delta}(x_0)$, with $\mathcal{H}^{n-1}(\Gamma_i \cap B_{\delta}(x_0)) < \infty$.  However, since $\Gamma_1$ might be badly behaved near free boundary points $x_0$ at which $\nabla \varphi(x_0) = 0$, we cannot conclude that $\mathcal{H}^{n-1}(\Gamma \cap K) < \infty$ for all compact subsets $K \subset B_1$.
\end{rem}


\begin{thebibliography}{00}

\bibitem[ALS15]{ALS15} J.~Andersson, E.~Lindgren, A.~Shahgholian, \emph{Optimal regularity for the obstacle problem for the $p$-Laplacian}, J. Differential Equations 259 (2015), 2167--2179.

\bibitem[Caf77]{Caf77} L.~Caffarelli, \emph{The regularity of free boundaries in higher dimensions}, Acta Math. 139 (1977), 155--184.

\bibitem[Caf98]{Caf98} L.~Caffarelli, \emph{The obstacle problem revisited}, J. Fourier Anal. Appl. 4 (1998), 383-402.

\bibitem[CLR12]{CLR12}  S.~Challal, A.~Lyaghfouri, J.~F.~Rodrigues, \emph{On the $A$-obstacle problem and the Hausdorff measure of its free boundary}, Ann. Mat. Pura Appl. 191 (2012), 113-165.

\bibitem[CLRT14]{CLRT14}  S.~Challal, A.~Lyaghfouri, J.~F.~Rodrigues, R.~Teymurazyan, \emph{On the regularity of the free boundary for quasilinear obstacle problems}, Interfaces Free Bound. 16 (2014), 359–--394.

\bibitem[KKPS00]{KKPS00} L.~Karp, T.~Kilpel\"ainen, A.~Petrosyan, H.~Shahgholian, \emph{On the Porosity of Free Boundaries in Degenerate Variational Inequalities}, J. Differential Equations 164 (2000), 110--117.

\bibitem[LS03]{LS03} K.~Lee, A.~Shahgholian, \emph{Hausdorff measure and stability for the $p$-obstacle problem ($2<p<\infty$)}, J. Differential Equations 195 (2003), 14-24.

\bibitem[PSU12]{PSU12} A.~Petrosyan, H.~Shahgholian, N.~Uraltseva, \emph{Regularity of free boundaries in obstacle-type problems}, volume 136 of Graduate Studies in Mathematics. American Mathematical Society, Providence, RI, 2012.

\bibitem[V\'{a}z84]{Vaz84} J.~L.~V\'{a}zquez. \emph{A strong maximum principle for some quasilinear elliptic equations.} Applied Mathematics and Optimization 12 (1984): 191-202.




\end{thebibliography}
\end{document}